\documentclass{amsart}
\usepackage{amssymb}
\usepackage{amsrefs}
\usepackage{amsfonts}
\usepackage{graphicx}
\usepackage{amsmath}
\usepackage{amssymb}
\usepackage{amsthm}
\usepackage{enumerate}
\usepackage[all, cmtip]{xy}

\newcommand{\Z}{\ensuremath{\mathbb{Z}}}
\newcommand{\Q}{\ensuremath{\mathbb{Q}}}
\newcommand{\Hom}{Hom}
\newcommand{\Ext}{Ext} %---- Ext yazmak icin
\newcommand{\Soc}{Soc}
\newcommand{\Tor}{Tor}
\newtheorem{thm}{Theorem}

\newtheorem{corollary}[thm]{Corollary}

\newtheorem{defi}{Definition}

\newtheorem{lem}[thm]{Lemma}

\newtheorem{proposition}[thm]{Proposition}
\newtheorem{remark}[thm]{Remark}

\date{}

\begin{document}

\author{Y{\i}lmaz Dur\u{g}un}
\address{\c{C}ukurova University, Department of Mathematics, 01330, Adana, Turkey.}
\email{ydurgun@cu.edu.tr}

\title{On subinjectivity domains of pure-injective modules}

\begin{abstract}
As an alternative perspective
on the injectivity of a pure-injective module, a pure-injective module $M$ is
said to be pi-indigent if its subinjectivity domain is smallest
possible, namely, consisting of exactly the absolutely pure modules.
A module $M$ is called subinjective relative to a module $N$ if for every
extension $K$ of $N $, every homomorphism $N \to M$ can be extended to a homomorphism
$K \to M$. The subinjectivity domain of the module $M$ is defined to be the class of
modules $N$ such that $M$ is $N$-subinjective.
Basic properties of
the subinjectivity domains of pure-injective modules and of pi-indigent modules  are studied.
The structure of a ring over which every pure-injective (resp. simple, uniform, indecomposable)  module is injective or subinjective relative only to the smallest possible family of modules
 is investigated.
This work is a natural continuation to recent papers that
have embraced the systematic study of the subinjective and  subprojective domains of modules.

\end{abstract}

\date{\today}
\maketitle
\section{Introduction} \label{section1}

Throughout, $R$ will denote an associative ring with identity.
As usual, we denote by $Mod-R$ the category of right $R$-modules.
Some recent work in module theory has focused on classical injectivity in order to consider the extent of injectivity of modules from a fresh perspective.
Recall
that a module $M$ is said to be $N$-subinjective if for every extension $K$ of $N$ and every
homomorphism $f: N \to M$ there exists a homomorphism $h: K \to M$
such that $h|_{N} = f$.  For a module $M$, the subinjectivity domain
of $M$, $\underline{\mathfrak{In}}^{-1}(M)$, is defined to be the
collection of all modules $N$ such that $M$ is $N$-subinjective,
that is $\underline{\mathfrak{In}}^{-1}(M)=\{N\in Mod-R|$ M is
N-subinjective$\}$ (see \cite{pinar}). It is clear that a module $M$ is injective if
and only if $\underline{\mathfrak{In}}^{-1}(M)=Mod-R$. If $N$ is
injective, then $M$ is $N$-subinjective. So, the smallest
possible subinjectivity domain is the class of injective modules.
While traditionally the study of non-injective modules has emphasized those modules that are as injective as possible, the recent
 has been made to understand also the diametrical opposite notion of modules which are subinjective only with respect to the smallest possible class of modules, i.e. the class of injective modules.
 One of the notions thus introduced is that of an indigent module (see \cite{pinar}). Presently, it is not known whether indigent module exists for an arbitrary ring, but an affirmative answer is known for
 some rings, such as Noetherian rings \cite{f-indigent}. In \cite{noyanenginalizade}, the authors studied ring $R$ with the property that  every non-injective right module is indigent.
 Indigent modules have been recently studied  in \cite{on-indigent}.
 In \cite{pure-injective-indigent}, inspired
 by the notion of relative subinjectivity, pure-subinjectivity domain of a module is introduced and studied.
A module $M$ is said to be $N$-pure-subinjective if for every pure extension $K$ of $N$ and every
 homomorphism $f: N \to M$ there exists a homomorphism $h: K \to M$
 such that $h|_{N} = f$. The pure-subinjectivity domain of $M$ consists of those modules $N$ such that $M$ is $N$-pure-subinjective.

In this paper, motivated by these above-mentioned papers, we address some questions on pure-injective modules. The first question discussed here is: What is the structure of a ring  $R$ with the property that  every non-injective  pure injective left module is subinjective relative only to
the smallest possible class of modules?
To approach this problem, we use the following observation: The
smallest possible subinjectivity domain of a pure injective module is the class of absolutely pure (fp-injective) modules. To keep
in line with \cite{pinar}, we refer to these pure injective modules as pure injectively indigent, or pi-indigent for short. In contrast to indigent modules, such pure injective modules exist over any ring (Proposition \ref{existness}). We prove that if  such a ring is two-sided coherent then it is either two-sided semihereditary or right IF-ring (Proposition \ref{every module is flat or indigent}). We also prove that if  such a ring is left noetherian left nonsingular then every non-flat left $R$-module is i-test and  $R$ is left n-saturated ring (Lemma \ref{theorem 5}).

Second, we study the structure of a ring  $R$ with the property that  every non-injective pure injective right module is indigent.  We prove that if  such a ring is right noetherian then it is isomorphic to the direct product of a
semisimple Artinian ring and an indecomposable ring $A$ such that (i)  $A$ is right n-saturated matrix ring over local $QF$-ring; or, (ii) $A$ is hereditary Artinian serial ring with $J(A)^2=0$; or, (iii) $A$ is $SI$-ring with $\Soc(_{A}A)=0$ (Theorem \ref{MAiN THEOREM}).

Finally, we study the structure of a ring  $R$ with the property that  every non-injective simple right module is  indigent.  In \cite{noyanenginalizade},  Alizade, B{\"u}y{\"u}ka{\c{s}}ik and Er also investigate when non injective simple modules are indigents.  We improve their results proving that every non-injective simple right $R$-module is indigent  if and only if
 (i) $R$ is a right V-ring; or,
(ii)$R$ is right Hereditary righ Noetherian ring and, for any right $R$-module $M$, either $M$ is indigent or $Soc(M)=Soc(N)$, where $N$ is the  largest injective submodule of $M$; or,
(iii)$R\cong S\times T$, where $S$ is semisimple Artinian ring and $T$ is an indecomposable matrix ring over a local QF-ring (Theorem \ref{every simple is indigent}). We show that, for a right nonsingular ring which is not  right V-ring, every non-injective simple right $R$-module is indigent if and only if  every non-injective uniform (or indecomposable) right $R$-module with nonzero socle is indigent (Proposition \ref{every indecomposable is indigent}).
Moreover, we show that, for commutative nonsingular ring which is not V-ring, every non-injective simple module is indigent if and only if  every non-injective singular module is indigent (Proposition \ref{every singular is indigent}).

For a module $M$, $E(M)$, $Soc(M)$, and $Z(M)$ will stand for the injective hull, the socle and the singular submodule respectively.
 For any ring in our
discussion, $J$ will stand for the Jacobson radical of that ring.
For a module $M$, the character module $\Hom_{\Z}(M, \Q/\Z)$ is
denoted by $M^{+}$. For other concepts and problems not mentioned here,
For  other concepts or background
materials, we refer the reader to \cite{Rotman:HomologicalAlgebra, Relativehomologicalalgebra}.\

\section{The subinjectivity domain of a pure-injective module}\label{sec:basic properties}

The notion of purity has an significant role in module theory and model theory
since it was presented in the literature (see \cite{warfield-compactness, cohn-free}). There are several
generalizations of the notion of purity(see \cite{ simmons-cpure, warfield-compactness, puninsk-ring}).  For a survey on purities, we refer the reader to \cite{sklyarenko}.

Let $\delta$ be a class of right $R$-modules.
An exact sequence  $0\to A\to B\to C\to 0$  a sequence of left
$R$-modules is called $\delta$-pure,  if every member of $\delta$ is flat with respect to this sequence;
i.e., the induced homomorphism $M \otimes B\to M \otimes C$ is a monomorphism for
each $M \in \delta$. A submodule $A$ of an $R$-module $B$ is called a $\delta$-pure submodule if the
exact sequence $0\to A\to B\to C\to 0$ is $\delta$-pure. An $R$-module $M$ is said to
be $\delta$-pure injective (resp. projective) if $M$ has the injective (resp. projective) property relative to each $\delta$-pure exact
sequence. An $R$-module $N$ is said to
be absolutely $\delta$-pure (resp. $\delta$-flat) if every exact sequence starting (resp. ending) with $N$ is  $\delta$-pure. If  $\delta=\{M \}$, we say $M$-pure instead of $\delta$-pure.

If we take for $\delta$ the class of all (or even only finitely presented) right $R$-modules, then we get the classical purity, which is usualy called the Cohns purity. In this
case, $\delta$-pure exact, $\delta$-pure injective, $\delta$-pure projective, absolutely $\delta$-pure and $\delta$-flat are commonly called pure exact, pure injective, pure projective, absolutely pure (or fp-injective) and flat, respectively.

The author in \cite{pinar} show for a module $M$ to be $ N$-subinjective, one only needs to extend maps to $E(N)$. We have improved this result as follows.

\begin{lem}\label{pure subinjectivity domains}Let $M, N \in Mod - R$. Then the following conditions are equivalent.

\begin{enumerate}
\item [(1)]$M$ is $N$-subinjective.
\item[(2)] For every $f : N \to M$ and every monomorphism $g : N \to F$ with $F\in\underline{\mathfrak{In}}^{-1}(M)$, there
exists $h : F \to M$ such that $hg = f $.
\item[(3)] There is a monomorphism $g : N \to F$ with $F\in\underline{\mathfrak{In}}^{-1}(M)$, such that for every
$f : N \to M$, there exists $h : F \to M$ such that $hg = f $.
\end{enumerate}
\end{lem}

\begin{proof}
 $(1) \Rightarrow (2)$ and $(2) \Rightarrow (3)$ follows by \cite[Lemma 2.2]{pinar}. $(3) \Rightarrow (1)$ Let $f : N\to M$ be a homomorphism. By assumption, there is a monomorphism $g : N \to F$ with $F\in\underline{\mathfrak{In}}^{-1}(M)$ and a homomorphism $h : F \to M$ such that $hg = f $.

$$
\begin{xy}
   \xymatrix{
   0\ar[r] & N \ar[d]^{f}\ar[r]^{g} &   F\ar[d]^{\iota}\ar@{-->}[ld]^{h}   \\
    &M&   E(F)\ar[l]^{s}   }
\end{xy}
$$
Since $M$ is $F$-subinjective, there exists a homomorphism $s : E(F) \to M$ such that $s\iota=h$. It is clear that $\iota g$ is a monomorphism and $f=s\iota g$. Then $M$ is $N$-subinjective by \cite[Lemma 2.2]{pinar}.
\end{proof}

 The character module $N^{+}$ of  a right $R$-module $N$ is a pure injective left $R$-module \cite[Proposition 5.3.7]{Relativehomologicalalgebra}.
It is known that $(M\otimes N)^{+}\cong\Hom(N, M^{+})$ by the
 adjoint isomorphism, (see \cite{Rotman:HomologicalAlgebra}). Then, the following obsevation is clear.

\begin{proposition}\label{bycharacter} Let $M$ be a right $R$-module and $N$ a left $R$-module.
Then, $N$ is absolutely $M$-pure if and only if $M^{+}$ is $N$-subinjective.
\end{proposition}

Note that  a pure injective module $M$ is injective if
and only if $\underline{\mathfrak{In}}^{-1}(M)=Mod-R$. If $N$ is
absolutely pure, then $M$ is vacuously $N$-subinjective. So, the smallest
possible subinjectivity domain of a pure injective module is the class of absolutely pure modules.

\begin{proposition}\label{allmodules} Let $\textit{PI}$ be a class of all pure injective left $R$-modules. Then,  $\bigcap_{M\in \textit{PI}}\underline{\mathfrak{In}}^{-1}(M)=\{N\in R-Mod|N \text{ is absolutely pure}\}$
\end{proposition}

\begin{proof}
Let $K\in\bigcap_{M\in \textit{PI}}\underline{\mathfrak{In}}^{-1}(M)$. In particular, for each finitely presented right $R$-module $M$, $M^{+}$ is $K$-subinjective. By  the preceding proposition, $K$ is absolutely $M$-pure for each finitely presented right $R$-module $M$, and hence $K$ is absolutely pure module. The converse folows from the discussion in the preceding paragraph.
\end{proof}

As subinjectivity domains of pure-injective modules
clearly include all absolutely pure modules, a reasonable opposite to injectivity of pure-injective modules in this context is obtained by
considering pure-injective modules whose subinjectivity domain consists of only absolutely pure modules.
\begin{defi}We will call a pure-injective module $M$
 pure injectively indigent (or a pi-indigent) module in case
$\underline{\mathfrak{In}}^{-1}(M)=\{A\in Mod-R | \text{A is absolutely pure }\}$.
\end{defi}

Certainly, the first problem that comes to mind with the
introduction of the notion of pi-indigent modules is whether such pure injective modules
exist over all rings. For the remainder of this paper, let $\mathfrak{PI}:=\prod_{S_{i}\in\Gamma}S_{i}^{+}$,
where  $\Gamma$ be a complete set of representatives of finitely
presented left $R$-modules.

\begin{proposition}\label{existness}
$\mathfrak{PI}$ is a pi-indigent right $R$-module.
\end{proposition}
\begin{proof} A right $R$-module $N$ is absolutely pure if and only if $N\otimes M\to E(N)\otimes M$ is a monomorphism for each finitely presented left $R$-module $M$, i.e. $N$ is absolutely $M$-pure for each finitely presented left $R$-module $M$ (see, \cite[6.2.3]{Relativehomologicalalgebra}). Then,  $\mathfrak{PI}$ is a pi-indigent right $R$-module by Proposition \ref{bycharacter} and Lemma  \ref{pure subinjectivity domains}.
\end{proof}
Note that a ring $R$ is right Noetherian if and only if all absolutely pure right $R$-modules are injective (see \cite[Theorem 3]{megibben}). We have the following observation by Proposition \ref{existness}.
\begin{corollary}\label{indigent=noetherian}Let $R$ be a ring. $R$ is right Noetherian if and only if $\mathfrak{PI}$ is an indigent module.
\end{corollary}

A ring $R$ is right semihereditary if and only if every homomorphic image of an absolutely pure right $R$-module is
absolutely pure (see \cite[Theorem 2]{megibben}). In general, the subinjectivity domain of a pure-injective module is not closed with respect to
homomorphic images.
Consider for example the right $R$-module $\mathfrak{PI}$ over a ring $R$ which is not right semihereditary.
Since $R$ is not right semihereditary ring, the subinjectivity domain of  $\mathfrak{PI}$ is not closed with respect to
homomorphic images by Proposition \ref{existness}.

\begin{proposition}\label{semihereditaryring} A ring $R$ is right semihereditary if and only if the subinjectivity domain of any pure injective right $R$-module is closed under homomorphic images.
\end{proposition}

\begin{proof} Assume that $R$ is a right semihereditary ring. Suppose a pure injective right module $M$ is $N$-subinjective for a right module $N$.  Let $K$ be a submodule of $ N$ and $f: N/K\to M$ a homomorphism. Consider the following diagram:
$$
\begin{xy}
   \xymatrix{
    N \ar[d]^{\pi}\ar[r]^{\iota_{N}} &   E(N)\ar[d]^{\pi'}  \ar@{-->}[ldd]_{\varphi}  \\
    \frac{N}{K}\ar[d]^f\ar[r]&   \frac{E(N)}{K} \ar@{-->}[ld]^{\varphi{'}}\\
    M &
               }
\end{xy}
$$
where $\pi, \pi'$ are canonical epimorhisms. Since  $M$ is  $N$-subinjective, there is a $\varphi:E(N)\to M$ such that $\varphi_{|_N}=f\pi$. Clearly $f\pi(K)=0$, and so $K\subseteq Ker(\varphi)$. Then, by factor theorem, there exists a homomorphism $\varphi':E(N)/K\to M$ such that $\varphi'\pi'=\varphi$.
 Clearly, $\varphi'_{|_{N/K}}=f$ and, by semihereditary of $R$,  $\frac{E(N)}{K}$ is absolutely pure. Now, $M$ is $\frac{N}{K}$-subinjective by Lemma \ref{pure subinjectivity domains}. This proves the necessity.

Conversely suppose the subinjectivity domain of any pure injective right module is closed under homomorphic images. In particular, the subinjectivity domain of the module $\mathfrak{PI}$ is also closed under homomorphic image. But $\mathfrak{PI}$ is pi-indigent, and hence absolutely pure modules are closed under homomorphic image, this implies, by \cite[Theorem 2]{megibben}, that $R$ is a right semihereditary ring.
\end{proof}
A ring R is said to be a von Neumann regular ring if for each
$a \in R$ there is an $r \in R$ such that $a=  ara$. Every right (left) $R$-module is absolutely pure if and only
if $R$ is a von Neumann  regular ring. The proof of the following is obvious from the definitions.

\begin{proposition}\label{von neumann regular}
 The following statements are equivalent for a ring $R$:
 \begin{enumerate}
\item [(1)]$R$ is von Neumann regular.
\item [(2)]Every (non-zero) pure-injective right (left) $R$-module is
pi-indigent.
\item [(3)]There exists an injective  pi-indigent right (left) R-module.
\end{enumerate}
\end{proposition}

\textit{In the rest of this article, unless otherwise stated, all rings will be non von Neumann regular.}

\begin{proposition}\label{nonflat ideals are pi-indigent}Let $I$ be a non flat right ideal of a ring $R$. If  $I^{+}$ is pi-indigent, then $R$  is a left absolutely pure.
\end{proposition}
\begin{proof}Recall that a right $R$-module $M$ is flat if and only if $M^{+}$ is injective (\cite[Theorem 3.52]{Rotman:HomologicalAlgebra}). Therefore,  $I^{+}$ is not injective.
Assume that $I^{+}$ is pi-indigent. Note that $I^{+}$ is an epimorphic image of the injective module $R^{+}$.
But since $I^{+}$ is pi-indigent, $R$ must be left absolutely pure by \cite[Lemma 2.3]{p-indigent}.
\end{proof}

The weak global dimension of $R$, $wD(R)$, is less than or equal 1 if and only if every submodule of a flat right (left) $R$-module is flat if and only if every (finitely generated) right (left) ideal is flat, (see \cite[9.24]{Rotman:HomologicalAlgebra}).

\begin{corollary}\label{nonflat ideals are f-indigent}If  $I^{+}$ is injective or pi-indigent for every  finitely generated right ideal $I$ of $R$, then $R$  is a left absolutely pure or  $wD(R)\leq 1$.
\end{corollary}

\section{Every pure-injective module is injective or pi-indigent}\label{sec:everymodulef-indigent}

This section deals with  the structure of a ring  $R$ with the property that  every non-injective pure-injective  module is subinjective relative only to absolutely pure  modules.
Note that a pure-injective right module is injective if it is also absolutely pure. The following proposition states an obvious fact without proof. We will use this proposition freely in the sequel.

\begin{proposition}\label{P} The following conditions are equivalent for a ring $R$:

\begin{enumerate}
\item [(1)] Every pure-injective left module is injective or  pi-indigent;
\item [(2)] Every pure-injective left module is absolutely pure or  pi-indigent.
\end{enumerate}
\end{proposition}

For easy reference, (P) stands for the property that $R$ satisfies the equivalent conditions of Proposition \ref{P}.

 The following Proposition can be proved using standard arguments.  We omit its proof, which has much in common with the proof of \cite[6.2.3]{Relativehomologicalalgebra}.
\begin{proposition}\label{m-fp-injective} The following statements are
equivalent for any given modules $M_{R}$ and $_{R}N$.
\begin{enumerate}
\item [(1)]$_{R}N$ is absolutely $M_{R}$-pure.
\item [(2)]$M\otimes N\to M\otimes E(N)$ is a monomorphism.
\item [(3)] There exists an absolutely pure extension $E$ of $N$ such that $ M\otimes N\to M\otimes E$ is a monomorphism.
\end{enumerate}
\end{proposition}

\begin{proposition}\label{every module is flat or indigent}Let $R$ be a two-sided  coherent ring satisfying the condition (P).
Then $R$ is either $R$ is two-sided semihereditary or $R$ is right IF-ring.
\end{proposition}
\begin{proof} By Corollary \ref{nonflat ideals are f-indigent}, $R$  is  $wD(R)\leq 1$ or a left absolutely pure. In the former case, by the coherence of $R$, $R$ is a two-sided semihereditary ring. In the latter case, by \cite[Proposition 4.2]{stenstrom}, $R$ is right IF-ring, i.e. every absolutely pure right module is flat.
 \end{proof}

In \cite{p-poor}, Holston et al. are interested in the projective analog of the notion of
subinjectivity. Namely, a module $M$ is said to be $N$-subprojective
if for every epimorphism $g : B\to N$ and homomorphism $f : M \to
N$, then there exists a homomorphism $h : M \to B$ such that $gh =
f$. For a module $M$, the subprojectivity domain of $M$,
$\underline{\mathfrak{Pr}}^{-1}(M)$, is defined to be the collection
of all modules $N$ such that $M$ is $N$-subprojective, that is
$\underline{\mathfrak{Pr}}^{-1}(M)=\{N\in Mod-R|$ M is
N-subprojective$\}$.

\begin{remark}\label{auslandertranspoze}Let $M_{R}$ be a finitely presented module, that is, $M$
has a free presentation $F_{1}\overset{f}{\to} F_{0}\to M\to 0$
where $F_{0}$ and $F_{1}$ are finitely generated free modules. If we
apply the functor $\Hom_{R}(.,R)$ to this presentation, we obtain
the sequence
$$\xymatrix
{0\to M^{*}\to F_{0}^{*}\to F_{1}^{*}\to Tr(M)\to 0}$$ where $Tr(M)$
is the cokernel of the dual map $F_{0}^{*}\to F_{1}^{*}$. Note that,
$Tr(M)$ is a finitely presented left $R$-module. The left $R$-module
$Tr(M)$ is called the Auslander-Bridger transpose of the right
$R$-module $M$, (see \cite{auslander}).\end{remark}
\begin{proposition}\cite[Proposition 2.7]{f-indigent}\label{charecteronscoherent}
Let $M_{R}$ be a finitely presented module. The following properties
hold for any modules $N_{R}$ and $_{R}K$:
\begin{enumerate}
\item [(1)] $K$ is absolutely $M$-pure if and only if
$\Ext(Tr(M),K)=0$.
\item [(2)] $M$ is $N$-subprojective if and only if  $\Tor(N,\,
Tr(M))=0$.
\end{enumerate}
\end{proposition}

\begin{corollary}\label{characterof sub pro and inj}
Let $M_{R}$ be a finitely presented module. The following properties
hold for any modules $N_{R}$ and $_{R}K$:
\begin{enumerate}
\item [(1)] $K\in\underline{\mathfrak{In}}^{-1}(M^{+})$ if and only if $K^{+}\in\underline{\mathfrak{Pr}}^{-1}(M)$.

\item [(2)] $N\in\underline{\mathfrak{Pr}}^{-1}(M)$ if and only if $N^{+}\in\underline{\mathfrak{In}}^{-1}(M^{+})$.
\end{enumerate}
\end{corollary}
\begin{proof}
$(1)$ Let $K\in\underline{\mathfrak{In}}^{-1}(M^{+})$. Then
 $K$ is absolutely $M$-pure by Proposition \ref{m-fp-injective}, and so, by Proposition \ref{charecteronscoherent},
$\Ext(Tr(M),K)=0$. Since $\Tor(K^{+},\, Tr(M))\cong\Ext(Tr(M),K)^{+}$ by
 \cite[Theorem 9.51]{Rotman:HomologicalAlgebra},  $\Tor(K^{+},\, Tr(M))=0$. Hence, $K^{+}\in\underline{\mathfrak{Pr}}^{-1}(M)$ by Proposition \ref{charecteronscoherent}, and vice versa.

The proof of $(2)$ is similar to the proof of $(1)$.
\end{proof}

A module $M$ is called a Whitehead test
module for injectivity (i-test module)
 if $N$ is injective whenever $\Ext_{R}^{ 1}(M,N) =
0$ (\cite{trlifaj1}). A non-semisimple ring $R$ is said to be fully(resp. n-) saturated
provided that all non-projective (resp. finitely generated) modules are i-test. In \cite[Theorem
16]{noyanenginalizade}, authors proved that
a non von Neumann regular ring $R$ is right fully saturated if and
only if all non-injective modules are indigent.
\begin{lem}\label{theorem 5} Let $R$ be a  left noetherian left nonsingular ring which is satisfying the condition (P). Then, every non-flat left $R$-module is i-test. In particular, $R$ is left n-saturated ring.

 \end{lem}
 \begin{proof}Recall that a module is flat if and only if its character module is injective. Hence, the subprojectivity domain of any finitely presented right module under the assumption (P), consists precisely of the flat modules by Corollary \ref{characterof sub pro and inj}(2).

 By Corollary \ref{nonflat ideals are f-indigent}, $R$  is  left fp-injective or $wD(R)\leq 1$ . In the former case, $R$ is QF-ring by \cite[Theorem 2.2]{faith}, yielding a contradiction because $R$ is left nonsingular.
 In the latter case, by the noethernity of $R$, $R$ is left hereditary. Let $M$ be any non flat left $R$-module. $M$ has a finitely presented submodule which is not projective by \cite[Corollary 3.49]{Rotman:HomologicalAlgebra}, say $S$. Assume that $\Ext(S, N)=0$ for some left $R$-module $N$. Note that $Tr(Tr(S))\cong S$. By Proposition \ref{charecteronscoherent} and the adjoint isomorphism, $Tr(S)^{+}$ is $N$-subinjective. But, by the property (P), $Tr(S)^{+}$ is pi-indigent, and hence $N$ is injective. Therefore, $S$ is i-test, and hence  $M$ is also i-test by \cite[Proposition 4.3]{f-indigent}. Furthermore, by the noethernity of $R$, $R$ is left n-saturated ring.

 \end{proof}

  The projective analog of indigent modules was considered in \cite{p-poor}, namely, p-indigent
 modules. A module $M$ is p-indigent if
 $\underline{\mathfrak{Pr}}^{-1}(M)$ consists precisely of the
 projective modules.
  \begin{thm}\label{MAiN THEOREM}Let $R$ be a right Noetherian ring. Assume that every non-injective pure-injective left
  R-module is indigent. Then $R\cong S\times T$,
  where $S$ is a semisimple artinian ring and $T$ is an indecomposable ring satisfying
  one of the following conditions:
  \begin{itemize}
  \item[(1)] $T$ is right n-saturated matrix ring over local $QF$-ring; or,
  \item[(2)] $T$ is hereditary Artinian serial ring with $J(T)^2=0$; or,
  \item[(3)] $T$ is $SI$-ring with $\Soc(_{T}T)=0$.
  \end{itemize}
  \end{thm}
  \begin{proof}
  If every pure-injective left $R$-module is injective, then $R$ is von Neumann regular. Assume that there exists at least one non-injective pure injective indigent left $R$-module.  Since each pure-injective left $R$-module is $N$-subinjective for every absolutely pure left module $N$, all absolutely left $R$-modules are injective, yielding that $R$ is a left noetherian ring by \cite[Theorem 3]{megibben}.

  By Corollary \ref{nonflat ideals are f-indigent}, $R$ is left absolutely pure or left semihereditary ring. In the former case $R$ is $QF$-ring by \cite[Theorem 2.2]{faith}.  Let $M$ be a non-projective finitely generated right $R$-module. For a right $R$-module $N$, by Corollary \ref{characterof sub pro and inj}, $N\in\underline{\mathfrak{Pr}}^{-1}(M)$ if and only if $N^{+}\in\underline{\mathfrak{In}}^{-1}(M^{+})$. By our assumption, $M^{+}$ is indigent, and thus $N^{+}$ is injective. By \cite[Theorem 4]{Flatandprojectivecharactermodules}, $N$ is projective. Then, every finitely generated right R-module is projective or p-indigent.
  This implies, by \cite[Theorem 4.1]{f-indigent}, that there is a ring direct sum $R\cong S\times T$, where $S$ is semisimple Artinian ring and  $T$ is an indecomposable ring which is right n-saturated matrix ring over a local QF-ring.

   In the latter case, $R$ is left hereditary ring by the noetherianity.
    $R$ is not right $IF$-ring, because otherwise so would be $R$ is left fp-injective by \cite[Corollary 8.2]{faith}.
   Then, by noethernity of $R$, $R$ is both $QF$ and hereditary, and this implying $R$ is semisimple artinian, a contradiction. Therefore, $R$ has a non-flat injective right $R$-module, say $E$. The character module $E^{+}$ is a nonsingular left module by  \cite[Theorem 2]{Flatandprojectivecharactermodules} and  \cite[Proposition 4.4]{f-indigent}. Note that $E^{+}$ is not injective, because otherwise so would be $E$ is flat by \cite[Theorem 2]{Flatandprojectivecharactermodules}, a contradiction. Then, $E^{+}$ is indigent nonsingular module by our assumption, and thus
    $R$ is a left $SI$-ring, i.e. every singular left module is injective.

    Now, we will show that $R$ has a unique singular
       simple right $R$-module up to isomorphism.  Let $A$ be a non-projective simple right $R$-module. By noethernity of $R$, it is finitely presented. For a right $R$-module $N$, by Corollary \ref{characterof sub pro and inj}, $N\in\underline{\mathfrak{Pr}}^{-1}(A)$ if and only if $N^{+}\in\underline{\mathfrak{In}}^{-1}(A^{+})$. Note that $A^{+}$ is not injective, otherwise, by \cite[Theorem 3.52]{Rotman:HomologicalAlgebra},  $A$ is flat. But $A$ is finitely presented, so $A$ becomes projective, a contradiction.  Let $B$ be a singular
          simple right $R$-modules. Assume that $A$ and $B$ are not isomorphic. Then,
          $A$ is clearly $B$-subprojective, and so $B$ is flat. Since $R$ is
          right Noetherian, $B$ is finitely presented, and so it is projective by \cite[Corollary 3.58]{Rotman:HomologicalAlgebra}, contradicting the singularity of  $B$. Thus, $R$ has a unique singular simple right module $A$ up to
          isomorphism.

We claim that every finitely generated left $R$-module is a direct sum of a projective module and an injective module, i.e., $R$ is a left $FGPI$ ring (see \cite{structureofSIrings}). Let $M$ be a finitely generated  left $R$-module. Recall that $R$ is left Noetherian ring,  and so $M$ is finitely presented. Consider the exact sequence $0\to Z(M)\to M\to M/Z(M)\to 0$. Since $R$ is left $SI$-ring, $M\cong Z(M)\oplus (M/Z(M))$. $M/Z(M)$ is a nonsingular  finitely presented module. Since $R$ is left hereditary left noetherian ring, it is flat by \cite[Proposition 4.4]{f-indigent}, and so projective.
 Therefore, $R$ is a left $FGPI$ ring   By \cite[Theorem 8]{structureofSIrings}, $R\cong U\times V$ where $U$ is left Artinian left $SI$,
   $V$ is left and right $SI$-ring with $\Soc(_{V}V ) = 0$. If $A$ is right $U$-module, then $V$ must be zero, and hence $R\cong U$ is left Artinian. Then, $R$ is Artinian by the right noetherianity. By \cite[Theorem 4.1]{f-indigent},  $R\cong S\times T$, where $S$ is a semisimple artinian ring and $T$ satisfy (2) in Theorem \ref{MAiN THEOREM}. In case $S$ is $V$-module, $U$ is semisimple and $V$ satisfy (3) in Theorem \ref{MAiN THEOREM}.
  \end{proof}

 \section{Ring whose simple modules are indigent or injective}\label{sec:last}

 In \cite{noyanenginalizade}, the authors interested with the structure of rings over which every non-injective  module is
 indigent. In this section, we deal with the structure of rings over which every non-injective  simple (respectively, singular, uniform, indecomposable) module is indigent. Recall that a ring $R$ is called right $V$-ring if all simple right R-modules are injective.

 \begin{thm}\label{every simple is indigent} Let $R$ be a ring. The following are equivalent.
 \begin{enumerate}
 \item[(1)]Every non-injective simple right $R$-module is indigent.
 \item[(2)] One of the following statements hold:
  \begin{enumerate}
\item[(i)] $R$ is a right V-ring; or,
 \item[(ii)]$R$ is right Hereditary righ Noetherian ring and, for any right $R$-module $M$, either $M$ is indigent or $Soc(M)=Soc(N)$, where $N$ is the  largest injective submodule of $M$; or,
 \item[(iii)] $R\cong S\times T$, where $S$ is semisimple Artinian ring and $T$ is an indecomposable matrix ring over a local QF-ring.
  \end{enumerate}
 \end{enumerate}
 \end{thm}

 \begin{proof}$(1) \Rightarrow (2)$ Suppose every non-injective simple right $R$-module is indigent. If all simple modules are injective, then $R$ is a right  $V$-ring. Now suppose, there is a non-injective simple right $R$-module $U$. Then $U$ is indigent by the hypothesis. If $U'$ is any simple module which is not isomorphic to $U$, then $\Hom(U', U)=0$. That is, $U$ is $U'$-subinjective, and so $U'$ must be injective because $U$ is indigent. Thus the ring has a unique non-injective simple module up to  isomorphism, say $U$. Through injective modules, we have the following two cases:

 Case I. $\Hom(E,U)=0$ for each injective right $R$-module $E$. Then, $U$ is $N$-subinjective for any right $R$-module $N$ if and only if $\Hom(N,U)=0$. Let $N$ be a injective right $R$-module and let $K$ a submodule of $N$. $\Hom(N/K, U)=0 $, since $0\to \Hom(N/K, U)\to\Hom(N,U)=0$. Then $N/K$ is injective by indiginity of $U$. This implying $R$ is right Hereditary. Let $\{E_i\}_{i \in I}$ be an arbitrary family of injective modules. By  isomorphism $\Hom(\oplus_{i\in I}E_i, U)\cong\prod_{i\in I}\Hom(E_i,U)=0$, $U$ is $\oplus_{i\in I}E_i$-subinjective. As $U$ is indigent, $\oplus_{i\in I}E_i$ must be injective. So that the ring $R$ is Noetherian.

   Recall that a right module is called reduced if it has no nonzero injective submodule. Let $M$ be a right $R$-module. Since $R$ is right Hereditary and right Noetherian, $M=N \oplus M'$ for some $M', N \leq M$, where $M'$ is the reduced part of $M$ and  $N$ is the  largest injective submodule of $M$ (see \cite{matlis}). Note that, as $R$ is hereditary,  $\Hom(U', M')=0$ for every injective simple  right $R$-module. There are two cases for $M'$ by $U$: either $\Hom(U, M')=0$ or $\Hom(U, M')\neq 0$. In the former case, $Soc(M')=0$, and hence $Soc(M)=Soc(N)$. For the latter case, assume that $M'$ is not indigent. Then there is a non injective right $R$-module $K$ such that $M'$ is $K$-subinjective. $\Hom(K, U)\neq 0$, otherwise $K$ is injective by indiginity of $U$. Then $M'$ is $U$-subinjective by the fact that the subinjectivity domain of each right module is closed under homomorphic images over hereditary rings (see \cite[Theorem 2.1]{on-indigent}).
Since $R$ is right hereditary and $M'$ is reduced,  $\Hom(U, M')=0$, contradicting our assumption.
 Therefore, $M'$ is indigent, and so $M$ is.

 Case II. $\Hom(E,U)\neq 0$. Then $U$ is homomorphic image of an injective right $R$-module. By \cite[Lemma 2.3]{p-indigent}, $U$ is F-subinjective for every projective module $F$. As $U$ is indigent, every projective right $R$-module is injective, i.e. $R$ is QF-ring. Then, $U$ is a unique singular simple right  $R$-module. By \cite[Theorem 3.1]{p-indigent}, there is a ring direct sum $R\cong S\times T$, where $S$ is semisimple Artinian ring and $T$ is an indecomposable matrix ring over a local QF-ring.

 $(2) \Rightarrow (1)$ In case $R$ is a right V-ring, every simple right are module is injective, and
 we are done. For case 2-(ii), it is obvious. Assuming 2-(iii). Then $R$ has a unique  non-injective simple right $R$-module, say $U$. $U$ is indigent by \cite[Proposition 32]{noyanenginalizade}.

 \end{proof}
In the preceding Theorem, $2-(ii)$ is equivalent to say that $R$ is right Hereditary righ Noetherian ring with a unique non-injective  simple  module (up to
 isomorphism) and every reduced  right $R$-module $M$ with $Soc(M)\neq 0$ is indigent. Furthermore, I would point out that  a ring which has a non-injective simple indigent right module is right Noetherian.
 Note that over a commutative ring $R$ a simple module  $S$ is pure-injective and $S\cong S^{+}$ (see \cite{semsismple-chetman}). We have the following result by Theorem \ref{every simple is indigent} and \cite[Theorem 5.2]{f-indigent}.
 \begin{corollary}\label{theorem 2}Let $R$ be a commutative ring. The following statements are equivalent.
 \begin{enumerate}
 \item[(1)] Every simple module is injective or indigent.
 \item[(2)] $R$ is a V -ring, or $R\cong S\times T$, where $S$ is semisimple Artinian ring and $T$ is a DVR, or a local QF-ring.
 \end{enumerate}
 \end{corollary}

 \begin{proposition}\label{every indecomposable is indigent}Let $R$ be a right nonsingular  ring which is not right  $V$-ring.  The following are equivalent.
 \begin{enumerate}
 \item[(1)]Every simple module is indigent or injective.
 \item[(2)] Every uniform module with nonzero socle is indigent or injective.
 \item[(3)] Every indecomposable module with nonzero socle is indigent or injective.
 \end{enumerate}
 \end{proposition}
  \begin{proof} \begin{sloppypar}In the general case there is the following chain of inclusions: $\text{simple modules} \subsetneqq \text {uniform modules} \varsubsetneq \text{indecomposable modules}$. Thus, $(3) \Rightarrow (2)$ and $(2) \Rightarrow (1)$ are obvious. To $(1) \Rightarrow (3)$, let $M$ be an indecomposable non injective module with $Soc(M)\neq 0$. Note that, by Theorem \ref{every simple is indigent}, $R$ is right hereditary and right Noetherian ring with unique non injective indigent simple module, say $U$. Then $M$ must be reduced. We will show that $M$ is indigent by mimicking second paragraph of the proof of (2-(ii)) in Theorem \ref{every simple is indigent}. Note that $\Hom(V, M)= 0$ for each simple module $V$, because otherwise $V$ is a direct summand of $M$, this would contradict the fact $M$ is indecomposable.
Therefore, $\Hom(U, M)\neq 0$, because otherwise $Soc(M)= 0$, contradicting its choice.\end{sloppypar}
 Assume that $M$ is not indigent. Then there is a non injective right $R$-module $K$ such that $M$ is $K$-subinjective. $\Hom(K, U)\neq 0$, otherwise $K$ is injective by indiginity of $U$. Then $M$ is $U$-subinjective by the fact that the subinjectivity domain of each right module is closed under homomorphic images over hereditary rings (see \cite[Theorem 2.1]{on-indigent}).
  Since $R$ is right hereditary and $M$ is reduced,  $\Hom(U, M)$ must be zero, a contradiction.
   Thus, $M$ is indigent.
 \end{proof}
 \begin{proposition}\label{every noninjective simple is singular}A noninjective simple module  over commutative ring is singular.
   \end{proposition}
   \begin{proof}
    Let $U$ be a noninjective simple module. Recall that every simple module is either projective or singular. Assume that $U$ is a projective module. Then $U^{++}$ is injective because $U\cong U^{+}$ and the character module of a flat module is injective. $U$ is a pure submodule of $U^{++}$  by \cite[Proposition 5.3.9]{Relativehomologicalalgebra}. Then, $U$ is an injective module by the fact that pure injective absolutely pure modules are injective, a contradiction. Therefore, $U$ is singular.
   \end{proof}
 \begin{proposition}\label{every singular is indigent}Let $R$ be a commutative nonsingular  ring which is not $V$-ring.  The following are equivalent.
  \begin{enumerate}
  \item[(1)]Every simple module is indigent or injective.
  \item[(2)] Every singular module is indigent or injective.
  \end{enumerate}
  \end{proposition}
   \begin{proof} $(1) \Rightarrow (2)$ By  Theorem \ref{every simple is indigent},  $R$ is  hereditary Noetherian ring with unique noninjective indigent simple module, say $U$. By Proposition \ref{every noninjective simple is singular}, $U$ is (unique) singular. Let $M$ be a noninjective singular module. Note that a module is indigent if it has an indigent direct summand. We may assume, without loss of generality, that $M$ is reduced. Since  $R$ is hereditary Noetherian ring, $R$ is C-ring, i.e. every singular module has nonzero socle (see \cite[Proposition 5.4]{noetherhereditarycring}).  Then $\Hom(U, M)\neq 0$. By mimicking second paragraph of the proof of Proposition \ref{every indecomposable is indigent}, it is obtained that $M$ is indigent. $(2) \Rightarrow (1)$ follows by Proposition \ref{every noninjective simple is singular}.
  \end{proof}

 In \cite[Theorem 14]{noyanenginalizade}, it is shown that  $R$ is a QF-ring isomorphic to a matrix ring over a local ring
  if and only if R is right Artinian with homogeneous right socle containing $Z(R_{R})$ essentially, and
 every simple right module is indigent or injective. We have the following result by Theorem \ref{every simple is indigent}.
 \begin{corollary}Let $R$ be a ring which is not right $V$-ring.  Assume that $Z(R_{R})\neq 0$.  The following are equivalent.
 \begin{enumerate}
 \item[(1)]Every simple right $R$-module is indigent or injective.
 \item[(2)]  $R\cong S\times T$, where $S$ is semisimple Artinian ring and $T$ is an indecomposable matrix ring over a local QF-ring.
 \item[(3)]  The left-hand version of $(1)$.
 \end{enumerate}
 \end{corollary}

\centerline{\bf{Acknowledgments}} %I also wish to thank the anonymous referee for careful reading
%and valuable suggestions for improvement.
This work was supported by Research Fund of the \c{C}ukurova University (Project Number:10871).
\bibliography{mybibfile}

\begin{thebibliography}{aa}

\bibitem{noyanenginalizade}
R.~Alizade, E.~B{\"u}y{\"u}ka{\c{s}}ik, and N.~Er.
\newblock Rings and modules characterized by opposites of injectivity.
\newblock {\em J. Algebra}, 409:182--198, 2014.

\bibitem{on-indigent}
F.~Altinay, E.~B{\"u}y{\"u}ka{\c{s}}ik, and Y.~Dur\u{g}un.
\newblock On the structure of modules defined by subinjectivity.
\newblock {\em J. Algebra Appl.}, (accepted), 2018.

\bibitem{auslander}
M.~Auslander and M.~Bridger.
\newblock {\em Stable module theory}.
\newblock No. 94. American Mathematical Society, Providence, R.I., 1969.

\bibitem{pinar}
P.~Aydo{\u{g}}du and S.~R. L{\'o}pez-Permouth.
\newblock An alternative perspective on injectivity of modules.
\newblock {\em J. Algebra}, 338:207--219, 2011.

\bibitem{semsismple-chetman}
T.~J. Cheatham and J.~R. Smith.
\newblock Regular and semisimple modules.
\newblock {\em Pacific J. Math.}, 65(2):315--323, 1976.

\bibitem{Flatandprojectivecharactermodules}
T.~J. Cheatham and D.~R. Stone.
\newblock Flat and projective character modules.
\newblock {\em Proc. Amer. Math. Soc.}, 81(2):175--177, 1981.

\bibitem{cohn-free}
P.~M. Cohn.
\newblock On the free product of associative rings.
\newblock {\em Math. Z.}, 71:380--398, 1959.


\bibitem{p-indigent}
Y.~Dur\u{g}un.
\newblock Rings whose modules have maximal or minimal subprojectivity domain.
\newblock {\em J. Algebra Appl.}, 14(6):1550083, 12, 2015.

\bibitem{f-indigent}
Y.~Dur\u{g}un.
\newblock An alternative perspective on flatness of modules.
\newblock {\em J. Algebra Appl.}, 15(8):1650145, 18, 2016.

\bibitem{Relativehomologicalalgebra}
E.~E. Enochs and O.~M.~G. Jenda.
\newblock {\em Relative homological algebra}, volume~30.
\newblock Walter de Gruyter \& Co., Berlin, 2000.

\bibitem{faith}
C.~Faith and D.~V. Huynh.
\newblock When self-injective rings are {QF}: a report on a problem.
\newblock {\em J. Algebra Appl.}, 1(1):75--105, 2002.

\bibitem{p-poor}
C.~Holston, S.~R. L\'opez-Permouth, J.~Mastromatteo, and J.~E.
  Simental-Rodriguez.
\newblock An alternative perspective on projectivity of modules.
\newblock {\em Glasgow Math. J.}, 57(1):83--99, 2015.

\bibitem{structureofSIrings}
D.~V. Huynh.
\newblock Structure of some {N}oetherian {SI} rings.
\newblock {\em J. Algebra}, 254(2):362--374, 2002.


\bibitem{pure-injective-indigent}
S.~R. L\'{o}pez-Permouth, J.~Mastromatteo, Y.~Tolooei, and B.~Ungor.
\newblock Pure-injectivity from a different perspective.
\newblock {\em Glasg. Math. J.}, 60(1):135--151, 2018.

\bibitem{matlis}
E.~Matlis.
\newblock Injective modules over {N}oetherian rings.
\newblock {\em Pacific J. Math.}, 8:511--528, 1958.

\bibitem{noetherhereditarycring}
J.~C. McConnell and J.~C. Robson.
\newblock {\em Noncommutative {N}oetherian rings}, volume~30.
\newblock American Mathematical Society, Providence, RI, 2001.

\bibitem{megibben}
C.~Megibben.
\newblock Absolutely pure modules.
\newblock {\em Proc. Amer. Math. Soc.}, 26:561--566, 1970.

\bibitem{puninsk-ring}
G.~Puninski, M.~Prest, and P.~Rothmaler.
\newblock Rings described by various purities.
\newblock {\em Comm. Algebra}, 27(5):2127--2162, 1999.

\bibitem{Rotman:HomologicalAlgebra}
J.~Rotman.
\newblock {\em An Introduction to Homological Algebra}.
\newblock Academic Press, New York, 1979.

\bibitem{simmons-cpure}
J.~Simmons.
\newblock Cyclic-purity: a generalization of purity for modules.
\newblock {\em Houston J. Math.}, 13(1):135--150, 1987.

\bibitem{sklyarenko}
E.~G. Skljarenko.
\newblock Relative homological algebra in the category of modules.
\newblock {\em Uspehi Mat. Nauk}, 33(3(201)):85--120, 1978.

\bibitem{stenstrom}
B~Stenstr\"{o}m.
\newblock Coherent rings and fp-injective modules.
\newblock {\em J. London Math. Soc. (2)}, 2:323--329, 1970.

\bibitem{trlifaj1}
J.~Trlifaj.
\newblock Whitehead test modules.
\newblock {\em Trans. Amer. Math. Soc.}, 348(4):1521--1554, 1996.

\bibitem{warfield-compactness}
R.~B. Warfield, Jr.
\newblock Purity and algebraic compactness for modules.
\newblock {\em Pacific J. Math.}, 28:699--719, 1969.

\end{thebibliography}

\end{document}